\documentclass[12pt,reqno]{amsart}

\usepackage{amsfonts}
\usepackage{eurosym}
\usepackage{amssymb}
\usepackage{amsthm}
\usepackage{amsmath}
\usepackage{amsaddr}
\usepackage{bm}
\usepackage{cite}
\usepackage{mathrsfs}
\usepackage{xcolor}
\usepackage[OT1]{fontenc}
\usepackage[left=1.8cm, right=1.8cm, top=3cm]{geometry}
\usepackage{hyperref}
\usepackage{graphicx}

\hypersetup{colorlinks=true, linkcolor=blue, citecolor=red, urlcolor=blue}
\RequirePackage{times}
\allowdisplaybreaks
\newtheorem{theorem}{Theorem}[section]

\newtheorem{lemma}[theorem]{Lemma}

\newtheorem{remark}[theorem]{Remark}

\def\n{\textbf{\textit{n}}}
\def\R3{\mathbb{R}^3}
\def\F2o{\overline{F_2}}

\def\d{{\rm d}}

\def\L2{L^2(\Omega)}

\def \au {\rm}
\def \ti {\it}
\def \jou {\rm}
\def \bk {\it}
\def \no#1#2#3 {{\bf #1} (#3), #2.}
\def \eds#1#2#3 {#1, #2, #3.}
\def \nome#1#2 {{\bf #1}, (#2).}

\begin{document}
\title[On the nonlocal Cahn-Hilliard equation in three dimensions]{On the separation property and the global attractor for \\ the nonlocal Cahn-Hilliard equation in three dimensions}
\author{\textsc{Andrea Giorgini}}
\address{Politecnico di Milano\\ 
Dipartimento di Matematica\\
Via E. Bonardi 9, Milano 20133, Italy\\
\href{mailto:andrea.giorgini@polimi.it}{andrea.giorgini@polimi.it}
}


\begin{abstract}
In this note, we consider the nonlocal Cahn-Hilliard equation with constant mobility and singular potential in three dimensional bounded and smooth domains. Given any global solution (whose existence and uniqueness are already known), we prove the so-called {\it instantaneous} and {\it uniform} separation property: any global solution with initial finite energy is globally confined (in the $L^\infty$ metric) in the interval $[-1+\delta,1-\delta]$ on the time interval $[\tau,\infty)$ for any $\tau>0$, where $\delta$ only depends on the norms of the initial datum, $\tau$ and the parameters of the system. We then exploit such result to improve the regularity of the global attractor for the dynamical system associated to the problem. 
\end{abstract}

\maketitle
\tableofcontents

\section{Introduction and main results}

We study the nonlocal Cahn-Hilliard equation (see \cite{GL1997,GL1998, GGG}) 
\begin{equation}  \label{nCH}
\partial_t \phi = \Delta \left( F^{\prime}(\phi)-J\ast \phi \right) \quad 
\text{in } \Omega \times (0,\infty),
\end{equation}
where $\Omega$ is a smooth and bounded domain in $\mathbb{R}^3$. The state variable $\phi$ represents the difference of the concentrations of two fluids. This equation is commonly rewritten as 
\begin{equation}  \label{nCH-2}
\partial_t \phi = \Delta \mu, \quad \mu = F^{\prime}(\phi)-J\ast \phi  \quad 
\text{in } \Omega \times (0,\infty),
\end{equation}
which is equipped with the following boundary and initial conditions 
\begin{equation}  \label{nCH-mu}
\partial_\n \mu=0 \quad \text{on } \partial \Omega \times (0,T),\quad
\phi(\cdot,0)= \phi_0 \quad \text{in } \Omega,
\end{equation}
where $\n$ is the outward normal vector on $\partial \Omega$. The physically relevant form of the nonlinear function $F$ is given by the convex part of the Flory-Huggins (also Boltzmann-Gibbs entropy) potential
\begin{equation}
\label{f-log}
F(s)=\frac{\theta}{2}\bigg[(1+s)\ln
(1+s)+(1-s)\ln (1-s)\bigg],\quad s\in [-1,1].
\end{equation}
The function $J:\mathbb{R}^3 \to \mathbb{R}$ is a (sufficiently smooth) interaction kernel such that $J(x)=J(-x)$. The notation $(J \ast \phi) (x)$ stands for  $\int_{\Omega} J(x-y)\phi(y)\, \d y$. The system \eqref{nCH-2}-\eqref{nCH-mu} is a gradient flow with respect to the metric of $H_{(0)}^1(\Omega)'$, namely the dual of $H^1(\Omega)$ with zero mean value, associated to the free energy
\begin{equation} 
\label{nfree}
\begin{split}
E_{NL}(\phi )
&= -\frac{1}{2}\int_{\Omega \times \Omega }J(x-y) \phi(y)\phi(x)\, \mathrm{d}
x \, \mathrm{d} y + \int_\Omega F(\phi (x))\, \mathrm{d}x\\
&=\frac{1}{4}\int_{\Omega \times \Omega }J(x-y) \vert \phi(y) -
\phi(x)\vert^2\, \mathrm{d} x \, \mathrm{d} y + \int_\Omega F (\phi (x)) -%
\frac{a(x)}{2}\phi^2(x) \, \mathrm{d}x,
\end{split}
\end{equation}
where $a(x)=(J\ast 1)(x)= \int_{\Omega} J(x-y) \, \d y$ for $x\in \Omega$. The function $\mu$ appearing in \eqref{nCH-2} is the so-called chemical potential,  which corresponds to $\frac{\delta E_{NL}\left( \phi\right) }{\delta \phi}$.

The analysis of the nonlocal Cahn-Hilliard equation with logarithmic potential \eqref{Log} (actually a more general class of {\it singular} potentials) has been firstly studied in \cite{GGG} (see also \cite{FG} for another proof of existence and \cite{GG2014} for the viscous case).
In particular, the authors in \cite{GGG} proved the existence and uniqueness of global weak solutions and their propagation of regularity for positive times (see proof of Theorem \ref{SPHR} below for more details). Such solutions satisfy
\begin{equation}  \label{bound}
\phi \in L^{\infty }(\Omega \times (0,\infty))\text{ with }|\phi (x,t)| < 1%
\text{ for a.e. } x \in \Omega, \, \forall \, t >0.
\end{equation}
Such property has an important physical meaning since the solution $\phi$ takes value in the significant interval $[-1,1]$ (cf. definition of $\phi$). Concerning the regularity of the global solutions, a main task consists in establishing $L^p$ estimates of  $F^{\prime\prime}(\phi)$ and $F^{\prime\prime\prime}(\phi)$, which are needed to prove the existence of {\it classical} solutions. This is a difficult question due to the growth conditions 
\begin{equation}  \label{growth-bounds}
F^{\prime\prime}(s)\leq C\mathrm{e}^{C |F^{\prime}(s)|}, \quad
|F^{\prime\prime\prime}(s)|\leq C F^{\prime\prime}(s)^2,
\end{equation}
which prevent the possibility to control $F^{\prime\prime}(\phi)$ or $%
F^{\prime\prime\prime}(\phi)$ in $L^p$ spaces in terms of some $L^p$ norms
of $F^{\prime}(\phi)$ (as possible in the case of potential with polynomial growth). However, although $L^p$ estimates of $F^{\prime\prime}(\phi)$ and $F^{\prime\prime\prime}(\phi)$ can be useful, this is not sufficient (in many cases) to prove higher order regularity, and it is necesssary to show the {\it instanteneous} (also called {\it strict}) {\it separation property}: for any $\tau >0$, there exists $\delta=\delta(\tau) \in (0,1)$ such that
\begin{equation}
\left\vert \phi (x,t)\right\vert \leq 1-\delta ,\text{ for all }(x,t)\in
\Omega \times (\tau ,\infty).  \label{sep-prop}
\end{equation}
We point out that the separation property is expected due to the gradient flow structure of the Cahn-Hilliard model, which drives the dynamics
towards stationary states of the free energy consisting of \textit{separated}
functional minima. In \cite{GGG}, a first proof of \eqref{sep-prop} has been established in \cite[Theorem 5.2]{GGG} in the two dimensional case. The argument hinges upon an iterative Alikakos-Moser argument for the powers of $F^{\prime}(\phi)$ combined with Gagliardo-Nirenberg
interpolation inequalities and the Trudinger-Moser inequality. 
A new proof of such result admitting a more general class of singular potentials has been proposed in \cite[Section 4]{GGG2}. The latter relies on a De Giorgi's iterative argument. This method is usually employed to obtain an $L^\infty$ estimate of the solution to a second order PDE, thereby the main achievement in \cite{GGG2} was to recast the method in order to get a {\it specific} bound (cf. \eqref{sep-prop} with \eqref{bound}). More recently, the separation property has been proven in three dimensions in \cite{P}, which allowed to show the convergence to stationary states. The author in \cite{P} improved the method in \cite{GGG2} in two ways: the truncated functions $\phi_n$ (see proof of Theorem \ref{SPHR} below) are shown to be bounded by $2\delta$ (instead of $1$ as in \cite{GGG2}) and a Poincar\'{e} type inequality for time-dependent functions is employed to avoid the integrals of  $\overline{\phi_n}$ (see term $Z_2$ in \cite[Section 4]{GGG2}). However, a main drawback of the argument, which is due to the latter ingredient, is that the value of $\delta$ in \eqref{sep-prop} depends on the particular solution. More precisely, $\delta$ cannot be estimated only in terms of {\it norm} of the initial data and the parameters of the system. The purpose of this work is to demonstrate that the De Giorgi iterative scheme in \cite{GGG2} and the observation $\| \phi_n\|_{L^\infty}\leq 2\delta$ are sufficient to achieve \eqref{sep-prop} with a value $\delta$ which depends on $\tau$, the initial energy $E_{NL}(\phi_0)$) and the parameters of the system (e.g. $F, \Omega, J$). Beyond its intrinsic interest, this allows us to improve the regularity of the global attractor for the dynamical system associated to the system \eqref{nCH-2}-\eqref{nCH-mu}.

In order to present the main results of this work, let us formulate the assumptions for the admissible class of potentials:  

\begin{itemize}
\item[(\textbf{A1})] $F\in C \left( \left[ -1,1\right]\right) \cap
C^{2}\left( -1,1\right)$ such that $\lim_{\left\vert s\right\vert \rightarrow 1}F^{\prime}\left( s\right) =\pm \infty$ and $F^{\prime\prime}(s)\geq \theta>0$ for all $s\in (-1,1)$.


\item[(\textbf{A2})] There exists $\varepsilon_0 > 0$ such that $F''$
is monotone non-decreasing on $[1-\varepsilon_0,1)$ and non-increasing in 
$(-1,1+\varepsilon_0]$.

\item[(\textbf{A3})] There exist $\varepsilon_1 \in (0,\frac12)$ and $C_F \geq 1$ such that
\begin{equation}
\label{o-gr}
\frac{1}{F'(1-2\delta)} \leq \frac{C_F}{|\ln(\delta)|}, \quad \frac{1}{F'(-1+2\delta)}\leq \frac{C_F}{|\ln(\delta)|}, \quad \forall \, 0 < \delta \leq \varepsilon_1
\end{equation}
and
\begin{equation}
\label{o-gr2}
\frac{1}{F''(1-2\delta)}\leq C_F \delta, \quad 
\frac{1}{F''(-1+2\delta)}\leq C_F \delta, \quad \forall \, 0 < \delta \leq \varepsilon_1.
\end{equation}
\end{itemize}

\begin{remark}
The assumptions (\textbf{A1})-(\textbf{A3}) are satisfied by the convex part of the Flory-Huggins potential \eqref{f-log}. 
\end{remark}

The main result reads as follows

\begin{theorem}
\label{SPHR} Assume that (\textbf{A1})-(\textbf{A3}) hold. Let $J$ be $%
W_{\rm loc}^{1,1}(\mathbb{R}^{3})$ such that $J(x)=J(-x)$ for all $x\in 
\mathbb{R}^{3}$. Assume that $\phi _{0}\in L^{\infty }(\Omega )$ such that 
$\Vert \phi _{0}\Vert _{L^{\infty }(\Omega )}\leq 1$ and $|\overline{\phi _{0}%
}|=|\Omega|^{-1} \left|\int_\Omega \phi_0(x) \, \d x\right|<1$. Then, for any $\tau >0$, there exists $\delta \in (0,1)$ such that the unique global solution to 
\eqref{nCH-2}-\eqref{nCH-mu} satisfies 
\begin{equation}
\label{SP}
\left\vert \phi (x,t)\right\vert \leq 1-\delta ,\text{ for a.e. }(x,t)\in
\Omega \times [\tau ,\infty).
\end{equation}
In addition, there exists three positive constants $C_1, C_2, C_3$ and $\alpha \in (0,1)$ such that
\begin{equation}
\label{MU-infty}
\sup_{t\geq \tau} \| \mu(t)\|_{L^\infty(\Omega)} \leq C_1, \quad  \sup_{t\geq \tau} \| \partial_t \mu\|_{L^{2}(t,t+1;L^{2}(\Omega)}\leq C_2,
\end{equation}
and
\begin{equation}
\label{PHI-hol}
|\phi(x_1,t_1)- \phi(x_2,t_2)|\leq C_3 \left( |x_1-x_2|^\alpha+ |t_1-t_2|^{\frac{\alpha}{2}}\right),
\end{equation}
for any $(x_1,t_1), (x_2,t_2)\in \Omega_t=\overline{\Omega}\times [t,t+1]$, for any $t \geq \tau$. The values of $\delta$, $C_1, C_2, C_3$ and $\alpha$ only depend on $\tau$, $\delta$, the initial energy $E_{NL}(\phi_0)$, the initial mean $\overline{\phi_0}$ and the parameters of the system (i.e. $F$, $J$, $\Omega$). 
\end{theorem}

\begin{remark}
A combination of the separation property \eqref{SP} and the H\"{o}lder regularity \eqref{PHI-hol} gives the following stronger result
\begin{equation}
\label{SP-2}
\left\vert \phi (x,t)\right\vert \leq 1-\delta, \quad \forall \, (x,t) \in \overline{\Omega}\times [\tau,\infty).
\end{equation}
\end{remark}


As a direct consequence of Theorem \ref{SPHR}, we infer additional features of the longtime behavior of the solutions of system \eqref{nCH-2}-\eqref{nCH-mu}.
Let us introduce the dynamical system associated with problem \eqref{nCH-2}-\eqref{nCH-mu}. For any given $m \in (0,1)$, we define the phase space
\begin{equation}
\mathcal{H}_{m}=\left\{ \phi \in L^\infty(\Omega):
\|\phi\|_{L^{\infty}(\Omega)}\leq 1 \text{ and } -1+m \leq \overline{\phi}\leq 1-m \right\}  \label{defhk}
\end{equation}%
endowed with the metric
\begin{equation}
\mathbf{d}(\phi _{1},\phi _{2})=\Vert \phi _{1} -\phi _{2}\Vert_{L^2(\Omega)}.  \label{metric}
\end{equation}%
The pair $(\mathcal{H}_{m}, \mathbf{d})$ is a complete metric space.
Then, we define the map
\begin{equation*}
S(t):\mathcal{H}_{m}\rightarrow \mathcal{H}_{m},\quad
S(t)\phi_{0}=\phi (t),\quad \forall \ t\geq 0,
\end{equation*}%
where $\phi $ is the global (weak) solution (see \cite[Theorem 3.4]{GGG}) originating from the initial condition $\phi _{0}$. It was shown in \cite[Section 4]{GGG} that $(\mathcal{H}_{m},S(t))$ is a dissipative dynamical system and $S(t)$ is a closed semigroup on the phase space $\mathcal{H}_{m}$ (see \cite{PZ}). Furthermore, the existence of the global attractor $\mathcal{A}_m$  was proven in \cite[Theorem 4.4]{GGG}. In particular, it is shown that $\mathcal{A}_m$ is a bounded set in $\mathcal{H}_m\cap H^1(\Omega)$. 
Our next result is concerned with the regularity of the global attractor $\mathcal{A}_m$.

\begin{theorem}
\label{GA}
Let (\textbf{A1})-(\textbf{A3}) hold. Assume that $J\in W_{\rm loc}^{1,1}(\mathbb{R}^{3})$ such that $J(x)=J(-x)$ for all $x\in 
\mathbb{R}^{3}$. Consider the connected global attractor $\mathcal{A}_m$ associated with the dynamical system $(\mathcal{H}_{\kappa },S(t))$. Then, $\mathcal{A}_m \subset B_{L^\infty(\Omega)}(0,1-\delta)$ and is bounded in $C^{\alpha}(\overline{\Omega})$.
\end{theorem}

Before proceeding with the proofs of the main results, it worth presenting a wider picture about the validity of the separation property for other Cahn-Hilliard equations. First, we recall the nonlocal Cahn-Hilliard equation with non-constant degenerate mobility
\begin{equation}  \label{nCH-deg}
\partial_t \phi = {\rm div}\left( (1-\phi^2)\nabla \mu\right), \quad \mu = F^{\prime}(\phi)-J\ast \phi  \quad 
\text{in } \Omega \times (0,\infty),
\end{equation}
which is completed with \eqref{nCH-mu}. In this case, the separation property has been previously proven by \cite{LP} in both two and three dimensions (see also \cite{FGG}). Next, we consider the (local) Cahn-Hilliard equation \cite{Ca,CH,CH2} (see also \cite%
{E,Mbook}) with constant mobility 
\begin{equation}
\partial _{t}\phi =\Delta \left( -\Delta \phi +\Psi ^{\prime }(\phi )\right)
\quad \text{in }\Omega \times (0,T),  \label{CH}
\end{equation}%
subject to the classical boundary and initial conditions 
\begin{equation}
\partial _{\n}\phi =\partial _{\n} \Delta
\phi =0\quad \text{on }\partial \Omega \times (0,T),\quad \phi (\cdot
,0)=\phi _{0}\quad \text{in }\Omega, \label{CH-bc-o}
\end{equation}
where $\Psi $ is the Flory-Huggins
potential defined by 
\begin{equation}
\Psi (s)=F(s)-\frac{\theta _{0}}{2}s^{2}=\frac{\theta }{2}\bigg[(1+s)\ln
(1+s)+(1-s)\ln (1-s)\bigg]-\frac{\theta _{0}}{2}s^{2},\quad s\in \lbrack
-1,1], \label{Log}
\end{equation}%
with constant parameters $\theta $ and $\theta _{0}$ fulfilling the
conditions $0<\theta <\theta _{0}$. The Cahn-Hilliard
system \eqref{CH} is the gradient flow with respect to the $%
H_{(0)}^{1}(\Omega )^{\prime }$ metric of the total free energy 
\begin{equation}  \label{free}
E_{L}(\phi )=\int_{\Omega }\frac{1}{2}|\nabla \phi |^{2}+\Psi (\phi (x)) \, 
\mathrm{d}x.
\end{equation}%
The separation property \eqref{sep-prop} for \eqref{CH}-\eqref{CH-bc-o} was first established in \cite{DD} and \cite{MZ} in one and two dimensions, respectively. The argument has been subsequently simplified in \cite{GGM} and \cite{HW2021}. More recently, it was extended to a more general class of potential in \cite{GGG2}. In three dimensions, the separation property has been shown only in \cite{AW} on the time interval $[T_s,\infty)$, where $T_s$ cannot be computed explicitly (see also \cite{LP} for a class of singular potentials different from \eqref{Log}). However, it still remains a major challenge to demonstrate the separation property for \eqref{CH}-\eqref{CH-bc-o}  for all positive times in three dimensions. 
Finally, we mention some recent results regarding the nonlocal-to-local
asymptotics obtained in \cite{DST2020, DST2021, GS2022}, 
That is, the weak solution to the nonlocal Cahn-Hillliard equation converges to the weak solution of the local Cahn-Hilliard equation, under suitable conditions on the data of the problem and a rescaling of the interaction kernel $J$.

\section{Separation property and H\"{o}lder regularity}
\label{lin-pro2} \setcounter{equation}{0}

In this section we provide an improved proof of the separation property for the
nonlocal Cahn-Hilliard equation in three dimensional domains. Then, we derive  some consequences on the regularity of the solution. 

Let us first recall the following well known result.
\begin{lemma}
\label{Iter} Let $\lbrace y_n \rbrace_{n \in \mathbb{N}_0} \subset \mathbb{R}%
_+$ satisfy the relation 
\begin{equation*}
y_{n+1}\leq C b^n y_n^{1+\epsilon},
\end{equation*}
for some $C>0$, $b>1$ and $\epsilon>0$. Assume that $y_0 \leq C^{-\frac{1%
}{\epsilon}} b^{-\frac{1}{\epsilon^2}}$. Then, we have 
\begin{equation*}
y_n \leq y_0 b^{-\frac{n}{\epsilon}}, \quad \forall \, n \geq 1.
\end{equation*}
In particular, $y_n \rightarrow 0$ as $n \rightarrow \infty$.
\end{lemma}

\begin{proof}[Proof of Theorem \protect\ref{SPHR}]
Let us report the well-posedness results from \cite[Theorems 3.4 and 4.1]{GGG}: there exists a unique weak solution $\phi: \Omega \times [0,\infty) \to \mathbb{R}$ to the system \eqref{nCH-2}-\eqref{nCH-mu} satisfying 
\begin{equation}
\label{reg-0}
\begin{split}
& \phi \in L^{\infty }(\Omega \times (0,\infty)): |\phi (x,t)|<1 \, \text{a.e. in }\Omega, \, \forall \, t>0,\\ 
& \phi \in L_{\rm loc}^{2}(0,\infty;H^{1}(\Omega ))\cap H^1_{\rm loc}(0,\infty; H^1(\Omega)'), \\
& \mu \in L_{\rm loc}^{2}(0,\infty;H^{1}(\Omega )),\quad F^{\prime }(\phi )\in
L_{\rm loc}^{2}(0,\infty;H^{1}(\Omega )),
\end{split}%
\end{equation}%
such that 
\begin{align}
& \left\langle \partial _{t}\phi ,v\right\rangle +(\nabla \mu ,\nabla
v)=0\quad \forall \,v\in H^{1}(\Omega ),\,\text{a.e. in }(0,\infty),
\label{weak-n} \\
& \mu =F^{\prime }(\phi )-J\ast \phi \quad \text{a.e. in }\ \Omega \times
(0,\infty),  \label{mu}
\end{align}%
and $\phi (\cdot ,0)=\phi _{0}(\cdot )$ in $\Omega $. Furthermore, for any $\tau \in (0,1)$ 
\begin{align}
& \sup_{t\geq \tau }\Vert \partial _{t}\phi (t)\Vert _{(H^{1}(\Omega
))^{\prime }}+\sup_{t\geq \tau }\Vert \partial _{t}\phi \Vert
_{L^{2}(t,t+1;L^{2}(\Omega ))}\leq \frac{C_{0}}{\sqrt{\tau }},
\label{reg1-n} \\
& \sup_{t\geq \tau }\Vert \mu (t)\Vert _{H^{1}(\Omega )}+\sup_{t\geq \tau
}\Vert \phi (t)\Vert _{H^{1}(\Omega )}+\sup_{t\geq \tau } \|F'(\phi) \|_{H^1(\Omega)} + \sup_{t\geq \tau }\Vert \mu \Vert
_{L^{2}(t,t+1;H^{2}(\Omega ))}\leq \frac{C_{0}}{\sqrt{\tau }},
\label{reg2-n}\\
&\sup_{t\geq \tau } \| \nabla \mu\|_{L^q(t,t+1;L^p(\Omega))}+ \| \nabla \phi\|_{L^q(t,t+1;L^p(\Omega))} \leq C_1(\tau), \quad \text{where } \, \frac{3p-6}{2p}=\frac{2}{q},\,  \forall \, p \in [2,6],
\label{reg3-n}
\end{align}%
where the positive constant $C_{0}$ only depends on $E_{NL}(\phi _{0})$, $%
\overline{\phi _{0}}$, $\Omega $ and the parameters of the system. The positive constant $C_1(\tau)$ also depends on the same quantities as $C_0$, in addition to $\tau$. Furthermore, the constants $C_0$ and $C_1$ are uniformly bounded in $\overline{\phi_0}$ if $\overline{\phi_0}$ lies in a compact set of $(-1,1)$.

In the first part of the proof, we show the separation property \eqref{SP}. To this end, we now introduce the iteration scheme \`{a} la De Giorgi devised in \cite[Section 4]{GGG2}. Let $\tau>0$ be fixed. We consider three positive parameters $T$, $\widetilde{\tau}$ and $\delta$ such that $T-3\widetilde{\tau }\geq \frac{\tau}{2}$ and $\delta \in \left(0,\min\lbrace \frac{\varepsilon_0}{2}, \varepsilon_1\rbrace \right)$ (cf. assumption (\textbf{A2})-(\textbf{A3})). The precise value of $\widetilde{\tau}$ and $\delta $ will be chosen afterwards. We define two sequences 
\begin{equation}
\begin{cases}
t_{-1}=T-3\widetilde{\tau } \\ 
t_{n}=t_{n-1}+\frac{\widetilde{\tau }}{2^{n}}%
\end{cases}%
\quad \forall \,n\geq 0,\quad \text{and} \quad k_{n}=1-\delta -\frac{\delta }{2^{n}},\quad
\forall \,n\geq 0.
\end{equation}%
Notice that 
\begin{equation}
t_{-1}<t_{n}<t_{n+1}<T-\widetilde{\tau },\quad \forall \,n\geq 0,\quad \text{such that} \quad 
t_{n}\rightarrow t_{-1}+2\widetilde{\tau }=T-\widetilde{\tau }\quad \text{as 
}n\rightarrow \infty ,
\end{equation}%
and 
\begin{equation}
1-2\delta \leq k_{n}<k_{n+1}<1-\delta ,\quad \forall \,n\geq 0,\quad \text{such that} \quad
k_{n}\rightarrow 1-\delta \quad \text{as }n\rightarrow \infty .
\end{equation}%
For $n\geq 0$, we introduce $\eta _{n}\in C^{1}(\mathbb{R})$ such that 
\begin{equation}
\eta _{n}(t)=%
\begin{cases}
1,\quad & t\geq t_{n} \\ 
0,\quad & t\leq t_{n-1}%
\end{cases}%
\quad \text{and}\quad |\eta _{n}^{\prime }(t)|\leq 2\frac{2^{n}}{\widetilde{%
\tau }}.  \label{eta-def}
\end{equation}%
Next, for $n\geq 0$, we consider the function
\begin{equation*}
\phi _{n}(x,t)=\max \{\phi (x,t)-k_{n},0\}=(\phi -k_{n})_{+}.
\end{equation*}%
Consequently, we introduce the sets 
\begin{equation*}
I_{n}=[t_{n-1},T] \quad \text{and} \quad 
A_{n}(t)=\{x\in \Omega :\phi (x,t)-k_{n}\geq 0\},\quad \forall \,t\in I_{n}.
\end{equation*}%
If $t\in \lbrack 0,t_{n-1})$, we set $A_{n}(t)=\emptyset$. We observe that
we have 
\begin{equation}
I_{n+1}\subseteq I_{n},\quad \forall \,n\geq 0,\qquad I_{n}\rightarrow
\lbrack T-\widetilde{\tau },T]\quad \text{as }n\rightarrow \infty ,
\label{I-set}
\end{equation}%
and 
\begin{equation}
A_{n+1}(t)\subseteq A_{n}(t),\quad \forall \,n\geq 0,\,t\in I_{n+1}.
\label{A-set}
\end{equation}%
The last ingredient is
\begin{equation*}
y_{n}=\int_{I_{n}}\int_{A_{n}(s)}1\,\mathrm{d}x\,\mathrm{d}s,\quad \forall
\,n\geq 0.
\end{equation*}%
For any $n\geq 0$, we choose as test function $v=\phi _{n}\eta _{n}^{2}$ in %
\eqref{weak-n}. Integrating over $[t_{n-1},t]$, where $t_{n}\leq t\leq T$,
we obtain the relation
\begin{equation}
\label{EQ.1}
\int_{t_{n-1}}^{t}\left\langle \partial _{t}\phi ,\phi _{n}\,\eta
_{n}^{2}\right\rangle \,\mathrm{d}s+\int_{t_{n-1}}^{t}\int_{A_{n}(s)}
\nabla F'(\phi ) \cdot \nabla \phi _{n}\,\eta _{n}^{2}\,%
\mathrm{d}x\,\mathrm{d}s=\int_{t_{n-1}}^{t}\int_{A_{n}(s)}(\nabla J\ast \phi
)\cdot \nabla \phi _{n}\,\eta _{n}^{2}\,\mathrm{d}x\,\mathrm{d}s.
\end{equation}%
Since $F'(\phi)\in L^\infty(\tau,\infty; H^1(\Omega))$ and $|\lbrace x \in \Omega: |\phi(x,t)|=1 \rbrace|=0$ for all $t \geq \tau$, we deduce from \cite{MM} that $h_k(F'(\phi))\in L^\infty(\tau,\infty; H^1(\Omega)\cap L^\infty(\Omega))$, where 
$$
h_k: \mathbb{R}\to \mathbb{R}, \quad h_k(s)=
\begin{cases}
k, \quad &s\geq k,\\
s, \quad &s \in (-k,k),\\
k, \quad &s \leq -k,
\end{cases}
\quad \forall \, k \in \mathbb{N}.
$$
Then, it follows that $h_k(F'(\phi))\to F'(\phi)$ almost everywhere in $\Omega$ and for all $t \geq \tau$, and $\nabla(h_k(F'(\phi)))=F''(\phi)\nabla \phi \,  1_{\lbrace |F'(\phi)| <k \rbrace}(\cdot) \to F''(\phi) \nabla \phi$  almost everywhere in $\Omega$ and for all $t \geq \tau$. Thus, by the monotone convergence theorem, $\int_\Omega |F''(\phi(t))\nabla \phi(t)|^2 \, \d x\leq \lim_{k\rightarrow \infty} \| h_k(F'(\phi(t)))\|_{H^1(\Omega)}^2 = \|F'(\phi(t))\|_{H^1(\Omega)}^2<\infty$, for all $t\geq \tau$. As consequence, it is easily seen that $\nabla F^{\prime }(\phi )=F^{\prime
\prime }(\phi )\nabla \phi $ in distributional sense. Thanks to this, we rewrite \eqref{EQ.1} as 
\begin{equation*}
\int_{t_{n-1}}^{t}\left\langle \partial _{t}\phi ,\phi _{n}\,\eta
_{n}^{2}\right\rangle \,\mathrm{d}s+\int_{t_{n-1}}^{t}\int_{A_{n}(s)}F^{%
\prime \prime }(\phi )\nabla \phi \cdot \nabla \phi _{n}\,\eta _{n}^{2}\,%
\mathrm{d}x\,\mathrm{d}s=\int_{t_{n-1}}^{t}\int_{A_{n}(s)}(\nabla J\ast \phi
)\cdot \nabla \phi _{n}\,\eta _{n}^{2}\,\mathrm{d}x\,\mathrm{d}s.
\end{equation*}%
Notice that
\begin{equation*}
\int_{t_{n-1}}^{t}\left\langle \partial _{t}\phi ,\phi _{n}\eta
_{n}^{2}\right\rangle \,\mathrm{d}s=\frac{1}{2}\Vert \phi _{n}(t)\Vert
_{L^{2}(\Omega )}^{2}-\int_{t_{n-1}}^{t}\Vert \phi _{n}(s)\Vert
_{L^{2}(\Omega )}^{2}\,\eta _{n}\,\partial _{t}\eta _{n}\,\mathrm{d}s.
\end{equation*}%
Also, by the choice of $\delta$, the assumption (\textbf{A2}) and the fact $A_n(t) \subseteq A_0(t)$ for $t\geq t_{n-1}$, we have
\begin{equation*}
\int_{t_{n-1}}^{t}\int_{A_{n}(s)}F^{\prime \prime }(\phi )\nabla \phi \cdot
\nabla \phi _{n}\,\eta _{n}^{2}\,\mathrm{d}x\,\mathrm{d}s\geq F^{\prime
\prime }\left( 1-2\delta \right) \int_{t_{n-1}}^{t}\Vert \nabla \phi
_{n}\Vert _{L^{2}(\Omega )}^{2}\,\eta _{n}^{2}\,\mathrm{d}s.
\end{equation*}%
Thus, we end up with
\begin{align*}
\frac12 \|\phi _{n}(t)\|_{L^{2}(\Omega )}^{2}& +F'' \left( 1-2\delta \right) \int_{t_{n-1}}^{t} \|\nabla \phi _{n}\|_{L^{2}(\Omega )}^{2} \, \eta _{n}^{2} 
\, \d s \\
& \leq \underbrace{\int_{t_{n-1}}^{t}\int_{A_{n}(s)}(\nabla J\ast \phi
)\cdot \nabla \phi _{n}\eta _{n}^{2}\,\d x \, \d s}_{I_{1}}+
\underbrace{\int_{t_{n-1}}^{t}\|\phi _{n}(s)\| _{L^{2}(\Omega
)}^{2}\eta _{n}\partial _{t}\eta _{n}\,\d s}_{I_{2}}, \quad \forall \,  t \in [t_{n},T].
\end{align*}%
We now observe that 
$$
\sup_{x \in \Omega} |(\nabla J\ast \phi)(x)|
=\sup_{x \in \Omega}\left| \int_\Omega \nabla J(x-y) \phi(y) \, \d y\right|
\leq \sup_{x \in \Omega} \int_\Omega |\nabla J(x-y)| \, \d y
= \sup_{x \in \Omega} \int_{x-\Omega} |\nabla J(z)| \, \d z.
$$
Since $\Omega$ is bounded, there exists $M>0$ such that $\Omega \subseteq B_M(\mathbf{0})$. Also, $\mathrm{diam}(\Omega)<\infty$. Then, there exists $M_1$ such that the set $x-\Omega \subset B_{M_1}(\mathbf{0}) $  for any $x\in \Omega$. It follows that 
\begin{equation}
\label{nablaJ}
\| \nabla J \ast \phi\|_{L^\infty(\Omega)}\leq \int_{B_{M_1}(\mathbf{0})} |\nabla J(z)|\, \d z= \| \nabla J\|_{L^1(B_{M_1}(\mathbf{0}))}.
\end{equation}
For simplicity of notation, we will use $B_{M_1}$ to denote $B_{M_1}(\mathbf{0})$. A similar argument applies for $\| J \ast \phi\|_{L^\infty(\Omega)}$.
Concerning the first term $I_1$, we obtain as in \cite[Section 4]{GGG2} that
\begin{align*}
I_{1}& 
=\int_{t_{n-1}}^{t}\int_{A_{n}(s)}(\nabla J\ast \phi )\,\eta
_{n}\cdot \nabla \phi \,\eta _{n}\,\d x\,\d s \\
& \leq \frac{1}{2}F''(1-2\delta) \int_{t_{n-1}}^{t}\| \nabla \phi _{n}\|_{L^{2}(\Omega )}^{2}\,\eta_{n}^{2}\,\d s
+\frac{1}{2}\frac{1}{F''(1-2\delta)} \int_{t_{n-1}}^{t}\int_{A_{n}(s)}|\nabla J\ast \phi |^{2}\,\eta_{n}^{2} \, \d x \,\d s \\
& \leq \frac{1}{2}F''(1-2\delta) \int_{t_{n-1}}^{t}\| \nabla \phi _{n}\|_{L^{2}(\Omega )}^{2}\,\eta_{n}^{2}\, \d s
+\frac{1}{2}\frac{1}{F''(1-2\delta)}
\int_{t_{n-1}}^{t}\| \nabla J\ast \phi \|_{L^{\infty}(\Omega )}^{2} 
\int_{A_{n}(s)}1 \, \d x \, \d s \\
& \leq \frac{1}{2}F''(1-2\delta)
\int_{t_{n-1}}^{t}\Vert \nabla \phi _{n}\Vert _{L^{2}(\Omega )}^{2}\,\eta
_{n}^{2} \, \d s
+\frac{1}{2}\frac{\Vert \nabla J\Vert _{L^{1}(B_{M_1}
)}^{2}}{F''(1-2\delta)}\int_{I_{n}}%
\int_{A_{n}(s)}1\,\d x \, \d s \\
& \leq \frac{1}{2}F^{\prime \prime }\left( 1-2\delta \right)
\int_{t_{n-1}}^{t}\Vert \nabla \phi _{n}\Vert _{L^{2}(\Omega )}^{2}\,\eta
_{n}^{2}\,\mathrm{d}s+\frac{1}{2}\frac{\Vert \nabla J\Vert _{L^{1}(B_{M_1})}^{2}}{F^{\prime \prime }\left( 1-2\delta \right) } \, y_{n}.
\end{align*}%
This is actually a correction of the argument in \cite{GGG} and \cite{P} where $\Vert \nabla J\Vert _{L^{1}(\Omega)}$ appears in the estimate analogous to the one above, instead of $\Vert \nabla J\Vert _{L^{1}(B_{M_1})}$. 
In order to handle the term $I_2$, we recall the main observation in \cite{P}:
\begin{equation}
\label{PO}
0\leq \phi_n \leq 2\delta \quad \text{a.e. in } \Omega, \, \forall \, t \in [T-2\widetilde{\tau}, T].
\end{equation}
By exploiting \eqref{eta-def} and \eqref{PO}, we simply have
\begin{equation*}
I_{2}\leq \frac{2^{n+1}}{\widetilde{\tau }}\int_{t_{n-1}}^{t}\int_{A_{n}(s)}%
\phi _{n}^{2}\,\d x\,\d s\leq \frac{2^{n+1}}{\widetilde{\tau }}%
\int_{I_{n}}\int_{A_{n}(s)} (2\delta)^2 \,\mathrm{d}x\,\mathrm{d}s
=\frac{2^{n+3}}{\widetilde{\tau }} \delta^2 y_{n}.
\end{equation*}%
Collecting the above estimates together, we infer that 
\begin{equation}
\Vert \phi _{n}(t)\Vert _{L^{2}(\Omega )}^{2}+F^{\prime \prime }\left(
1-2\delta \right) \int_{t_{n-1}}^{t}\Vert \nabla \phi _{n}\Vert
_{L^{2}(\Omega )}^{2}\,\eta _{n}^{2}\,\mathrm{d}s\leq \frac{\Vert \nabla
J\Vert _{L^{1}(B_{M_1} )}^{2}}{F^{\prime \prime }\left( 1-2\delta \right) }%
y_{n}+ 2^4 \frac{2^{n}}{\widetilde{\tau }} \delta^2 y_{n},  \quad \forall \, t\in [t_{n},T].
\label{L2-n}
\end{equation}%
As a consequence, 
\begin{equation}
\max_{t\in I_{n+1}}\Vert \phi _{n}(t)\Vert _{L^{2}(\Omega )}^{2}\leq
X_{n},\quad F^{\prime \prime }\left( 1-2\delta \right) \int_{I_{n+1}}\Vert
\nabla \phi _{n}\Vert _{L^{2}(\Omega )}^{2}\,\mathrm{d}s\leq X_{n},
\label{X_n}
\end{equation}%
where
\begin{equation}
\label{X_n-def}
X_{n}:=2^{n}\max \left\lbrace\frac{\Vert \nabla J\Vert _{L^{1}(B_{M_1}
)}^{2}}{F''\left( 1-2\delta \right) },\frac{2^4 \delta^2}{\widetilde{%
\tau }}\right\rbrace y_{n}.
\end{equation}
Now, in light of (\textbf{A3}), we observe that $\frac{\Vert \nabla J\Vert _{L^{1}(B_{M_1})}^{2}}{F''\left( 1-2\delta \right)}\leq C_F \delta \Vert \nabla J\Vert _{L^{1}(B_{M_1})}^{2}$, thereby
\begin{equation}
\label{X_n-2}
X_{n}=2^{n} \frac{\Vert \nabla J\Vert _{L^{1}(B_{M_1}
)}^{2}}{F''\left( 1-2\delta \right) } \, y_n, \quad \text{provided that}\quad
 \widetilde{\tau}\geq \frac{2^4 \delta}{C_F \Vert \nabla J\Vert _{L^{1}(B_{M_1})}^{2}}.
\end{equation}
The latter constraint will be verified later on. 

Next, for $t\in I_{n+1}$ and for almost every $x\in A_{n+1}(t)$, following \cite[Section 4]{GGG2}, we observe that 
\begin{align*}
\phi _{n}(x,t)& =\phi (x,t)-\left[ 1-\delta -\frac{\delta }{2^{n}}\right] \\
& =\underbrace{\phi (x,t)-\left[ 1-\delta -\frac{\delta }{2^{n+1}}\right] }%
_{=\phi _{n+1}(x,t)\geq 0}+\delta \left[ \frac{1}{2^{n}}-\frac{1}{2^{n+1}}%
\right] \geq \frac{\delta }{2^{n+1}},
\end{align*}%
which implies that
\begin{equation}
\begin{split}
\int_{I_{n+1}}\int_{\Omega }|\phi _{n}|^{\frac{10}{3}}\,\mathrm{d}x\,\mathrm{d}s& \geq
\int_{I_{n+1}}\int_{A_{n+1}(s)}|\phi _{n}|^{\frac{10}{3}}\,\mathrm{d}x\,\mathrm{d}s \\
& \geq \left( \frac{\delta }{2^{n+1}}\right)
^{\frac{10}{3}}\int_{I_{n+1}}\int_{A_{n+1}(s)}1\,\mathrm{d}x\,\mathrm{d}s=\left( \frac{%
\delta }{2^{n+1}}\right) ^{\frac{10}{3}}y_{n+1}.
\end{split}
\label{m-pow-new}
\end{equation}
In order to proceed with the next step, we recall the following Gagliardo-Nirenberg inequality in three dimensions
\begin{equation}
\label{GN}
\| u\|_{L^\frac{10}{3}(\Omega)}\leq C_{\Omega} \| u\|_{L^2(\Omega)}^\frac25 \| u\|_{H^1(\Omega)}^\frac35, \quad \forall \, u \in H^1(\Omega).
\end{equation}
Exploiting the definition of $y_{n}$, %
\eqref{I-set} and \eqref{GN}, we have
\begin{align*}
y_{n+1}\left( \frac{\delta }{2^{n+1}}\right) ^{\frac{10}{3}}
& \leq \int_{I_{n+1}}\int_{A_{n}(s)}|\phi _{n}|^{\frac{10}{3}}\,\mathrm{d}x\,%
\mathrm{d}s \\
&\leq C_\Omega  \int_{I_{n+1}} \| \phi_n\|_{L^2(\Omega)}^\frac43 \left( \| \nabla \phi_n\|_{L^2(\Omega)}^2 + \| \phi_n\|_{L^2(\Omega)}^2 \right) \, \d s\\
&\leq C_\Omega \underbrace{\int_{I_{n+1}} \| \phi_n\|_{L^2(\Omega)}^\frac43 \| \nabla \phi_n\|_{L^2(\Omega)}^2  \, \d s}_{A}
+ C_\Omega \underbrace{\int_{I_{n+1}} \| \phi_n\|_{L^2(\Omega)}^\frac43  \| \phi_n\|_{L^2(\Omega)}^2  \, \d s}_{B}.
\end{align*}%
As in \cite{GGG2}, we infer from \eqref{X_n} that
\begin{equation*}
A\leq \frac{1}{F^{\prime \prime }\left( 1-2\delta \right) }\max_{t\in
I_{n+1}}\Vert \phi _{n}(t)\Vert _{L^{2}(\Omega )}^{\frac43}F^{\prime \prime
}\left( 1-2\delta \right) \int_{I_{n+1}}\Vert \nabla \phi _{n}\Vert
_{L^{2}(\Omega )}^{2}\,\mathrm{d}s\leq \frac{1}{F^{\prime \prime }(1-2\delta
)}X_{n}^{\frac53}.
\end{equation*}%
On the other hand, by using \eqref{I-set} and \eqref{PO}, we notice that
\begin{align*}
B& \leq \max_{t\in
I_{n+1}}\Vert \phi _{n}(t)\Vert _{L^{2}(\Omega )}^{\frac43} \int_{I_{n}} \| \phi_n\|_{L^2(\Omega)}^2 \, \d s 
\leq (2\delta)^2 X_n^{\frac23}  \int_{I_{n}} \int_{A_n(s)} 1 \, \d x \, \d s 
= (2\delta)^2 X_n^{\frac23} y_n.
\end{align*}%
Thus, thanks to \eqref{X_n-2}, and making use of (\textbf{A3}), we find 
\begin{align*}
y_{n+1}\left( \frac{\delta }{2^{n+1}}\right) ^{\frac{10}{3}}
&\leq \left[ \frac{C_\Omega \| \nabla J\|_{L^1(B_{M_1})}^\frac{10}{3}}{(F''(1-2\delta))^\frac{8}{3}} \, 2^{\frac53 n} + \frac{4 C_\Omega \delta^2 \| \nabla J\|_{L^1(B_{M_1})}^\frac43}{(F''(1-2\delta))^\frac{2}{3}}\, 2^{\frac23 n} \right] y_n^\frac53\\
&\leq 
4 C_\Omega C_F^\frac83  \underbrace{\max\left\lbrace \| \nabla J\|_{L^1(B_{M_1})}^\frac{10}{3}, \| \nabla J\|_{L^1(B_{M_1})}^\frac{4}{3}\right\rbrace}_{C_J}   \, \delta^\frac83 \, 2^{\frac53 n} y_n^\frac53,
\end{align*}
which is equivalent to
\begin{align*}
y_{n+1}
&\leq 
\frac{2^\frac{16}{3} C_\Omega C_F^\frac83  C_J}{\delta^\frac23} \, 2^{5 n} y_n^\frac53.
\end{align*}
An application of Lemma \ref{Iter} with
$$
C= \frac{2^\frac{16}{3} C_\Omega C_F^\frac83  C_J}{\delta^\frac23}, 
\quad b=2^5, \quad \epsilon=\frac23
$$
entails that $y_n \rightarrow 0$ provided that
\begin{equation}
\label{y_0}
y_0 \leq \frac{\delta}{2^8 C_\Omega^\frac32 C_F^4 C_J^\frac32} \frac{1}{2^\frac{45}{4}}= \frac{\delta}{2^\frac{77}{4} C_\Omega^\frac32 C_F^4 C_J^\frac32}.
\end{equation}
We conclude from $y_n\rightarrow 0$ and
$
y_{n}\rightarrow \Big|\Big\lbrace(x,t)\in \Omega \times \lbrack T-\widetilde{%
\tau },T]:\phi (x,t)\geq 1-\delta \Big\rbrace\Big|$,  as $n\rightarrow \infty$, that
\begin{equation}
\label{SP-prep}
\Vert (\phi -(1-\delta ))_{+}\Vert _{L^{\infty }(\Omega \times (T-\widetilde{%
\tau },T))}=0.
\end{equation}
We are left to show that \eqref{y_0} is satisfied. Recalling (\textbf{A3}), \eqref{reg2-n} and $y_0= \int_{T-3\widetilde{\tau }}^{T}\int_{A_{0}(s)}1 \, \mathrm{d}x \, \mathrm{d}s$, we notice that (cf. \cite{GGG2,P}) 
\begin{equation}
\begin{split}
\int_{T-3\widetilde{\tau }}^{T}\int_{A_{0}(s)}1\,\mathrm{d}x\,\mathrm{d}s&
\leq \frac{\int_{T-3\widetilde{\tau }}^{T}\Vert F^{\prime}(\phi
(s))\Vert _{L^{1}(\Omega )} \, \mathrm{d}s}{|F^{\prime
}(1-2\delta )|} \\
& \leq  3 \widetilde{\tau} \Vert F^{\prime }(\phi )\Vert _{L^\infty(\frac{\tau}{2}, \infty; L^{1}(\Omega))} 
\frac{C_F}{|\ln(\delta)|}
= \frac{3 C_F \, C(E_{NL}(\phi_0),\tau) \widetilde{\tau}}{|\ln (\delta)|}.
\end{split}
\label{y0-bound}
\end{equation}
Thus, we impose that 
\begin{equation}
\label{tau-t-2}
\frac{3 C_F \, C(E_{NL}(\phi_0), \tau) \widetilde{\tau}}{|\ln (\delta)|} \leq \frac{\delta}{2^\frac{77}{4} C_\Omega^\frac32 C_F^4 C_J^\frac32}.
\end{equation}
In light of \eqref{X_n-2} and \eqref{tau-t-2}, we choose $\delta$ sufficiently small such that $\widetilde{\tau}$ satisfies the relations
\begin{equation}
\label{delta-def}
\frac{2^4 \delta}{C_F \Vert \nabla J\Vert _{L^{1}(B_{M_1})}^{2}} \leq
\widetilde{\tau} \leq \frac{\delta |\ln(\delta)|}{3 \, 2^\frac{77}{4} C_\Omega^\frac32 C_F^5 C_J^\frac32 C(E_{NL}(\phi_0), \tau)}.
\end{equation}
Now, set $T=\tau+\frac{\widetilde{\tau}}{2}$. Up to eventually reducing $\delta$ to get $\widetilde{\tau}$ even smaller, we clearly have $\tau-\frac{5\widetilde{\tau}}{2}\geq \frac{\tau}{2}$. Therefore, by \eqref{SP-prep}, we deduce that
$\Vert (\phi -(1-\delta ))_{+}\Vert _{L^{\infty }(\Omega \times (\tau-\frac{\widetilde{\tau}}{2},\tau+\frac{\widetilde{\tau}}{2}))}=0$. We point out that the value of $\widetilde{\tau}$ is independent of the choice of $T$. Thus, repeating the same argument on intervals of size $\widetilde{\tau}$, we conclude that  $\Vert (\phi -(1-\delta ))_{+}\Vert _{L^{\infty }(\Omega \times (\tau-\frac{\widetilde{\tau}}{2}, \infty)}=0$.
Finally, repeating the same argument for $(\phi+(-1+\delta))_{-}$, we arrive at the desired conclusion \eqref{SP}. It is important to highlight that the value of $\delta$ only depends on $F$, $J$, $\Omega$, $E_{NL}(\phi_0)$ and $\tau$.

The rest of the proof is devoted to the additional regularity results \eqref{MU-infty} and \eqref{PHI-hol} that can be inferred once the separation property is established. Firstly, by definition of $\mu$ in \eqref{nCH-2}, we observe that 
$$
\sup_{t\geq \tau} \| \mu (t)\|_{L^\infty}\leq \sup_{t\geq \tau} \left( \| F'(\phi(t))\|_{L^\infty(\Omega)} + \| J\ast \phi(t)\|_{L^\infty(\Omega)}\right) 
\leq |F'(1-\delta)| + \| J\|_{L^1(B_{M_1})}=:C_1.
$$
Let us observe that
\begin{equation}
\partial _{t}^{h}\mu(\cdot) =\partial _{t}^{h}\phi(\cdot) \left( \int_{0}^{1}F^{\prime
\prime }(s \phi(\cdot +h) +(1-s)\phi(\cdot) )\,\mathrm{d}s\right) -J\ast \partial
_{t}^{h}\phi(\cdot) ,\quad 0<t\leq T-h.  \label{a}
\end{equation}%
By \eqref{SP}, $\Vert s \phi(\cdot +h) +(1-s)\phi(\cdot)  \Vert
_{L^{\infty }(\Omega \times (\tau,\infty))}\leq 1-\delta $ for all $s\in (0,1)$.
Then, exploiting that $\Vert \partial
_{t}^{h}\phi \Vert _{L^{2}(0,T-h;L^{2}(\Omega ))}\leq \Vert \partial
_{t}\phi \Vert _{L^{2}(0,T;L^{2}(\Omega ))}$, we infer from \eqref{reg1-n} that
$\sup_{t\geq \tau}
\Vert \partial
_{t}^{h}\mu \Vert _{L^{2}(t,t+1;L^{2}(\Omega ))}\leq C_2,
$
where $C_2>0$ depends on $C_0$, $\tau$, $\delta$ and $J$, but is independent of $h$, . This implies that $\partial _{t}\mu \in L^{2}(0,T;L^{2}(\Omega ))$ for any $T>0$, and $\sup_{t\geq \tau} \| \partial_t \mu\|_{L^{2}(t,t+1;L^{2}(\Omega)}\leq C_2$.

Secondly, we study the H\"{o}lder continuity in both time and space. We notice that \eqref{nCH} is a quasi-linear equation with principal part in divergence form. Following the notation in the book \cite{LSU}, we define 
$a_l(x,t,u,p)=\widetilde{F}''(u) p_l -(\partial_l J \ast \phi(\cdot,t))(x)$, where $\widetilde{F}$ is the restriction of $F$ in $[-1+\delta,1-\delta]$. In light of the convexity of $F$ and $|\widetilde{F}''(s)|\leq |F''(1-\delta)|$, for all $s\in [-1+\delta,1-\delta]$, we deduce that
$$
a_l(x,t,u,p)p_l\geq \frac{\theta}{2} |p|^2 - \frac{1}{2\theta}\| \nabla J\|_{L^1(B_{M_1)}}, \quad |a_l(x,t,u,p)|\leq |F''(1-\delta)| |p| + \| \nabla J\|_{L^1(B_{M_1)}}.
$$
We also note that the solution $\phi$ satisfying \eqref{reg-0}-\eqref{reg3-n} is a {\it bounded generalized solution} in the sense of \cite[Chapter V]{LSU}. Thus, by \cite[Theorem 1.1, Chapter V]{LSU}, we deduce that \eqref{PHI-hol} holds in $\Omega'\times [t,t+1]$ , for any $t\geq \tau$ and $\Omega'\subset \Omega$ separated from $\partial \Omega$. In order to achieve \eqref{PHI-hol} up to the boundary, we make use of \cite[Corollary 4.2]{Dung}, which provides the desired conclusion under the same assumptions. It is worth noticing that the constant $C_3$ and the parameter $\alpha$ from both \cite{LSU} and \cite{Dung} only depends on $\delta$, $\theta$, $\| \nabla J\|_{L^1(B_{M_1})}$ and $\Omega$. 
This completes the proof.
\end{proof}

\section{On the regularity of the global attractor}

This section is devoted to some regularity properties of the global attractor $\mathcal{A}_m$ for the dynamical system $(\mathcal{H}_m, S(t))$ stated in Theorem \ref{GA}.

\begin{proof}[Proof of Theorem \ref{GA}] 
Let us consider $\phi^\star \in \mathcal{A}_m$. It is clear that $\| \phi^\star\|_{L^\infty(\Omega)}\leq 1$ such $|\overline{\phi^\star}|\leq 1-m$ and $\| \phi^\star\|_{H^1(\Omega)}\leq N_1$, where $N_1$ is a universal constant (namely, it does not depend on $\phi^\star$). We observe that $|E_{NL}(\phi^\star)|\leq N_2$, where $N_2$ is a universal constant depending only on $\| J\|_{L^1(B_{M_1})}$ (cf. \eqref{nablaJ}) and $\max_{s\in [-1,1]} |F(s)|$. 
Then, applying Theorem \ref{SPHR}, we deduce that 
\begin{equation}
\| S(t)\phi^\star\|_{L^\infty(\Omega)}\leq 1-\delta, \quad \forall \, t \geq [1,\infty).
\end{equation}
Here, $\delta$ depends on the constants in \eqref{delta-def}. In particular, since $|\overline{\phi^\star}|\leq 1-m$, it is easily seen that $C(E_{NL}(\phi^\star),1)\leq N_3$, where $N_3$ is a universal constant. This implies that $\delta$ is a universal constant. Thanks to the arbitrary of $\phi^\star$ in the above argument, we deduce that
$$
\mathcal{A}_m=S(1)\mathcal{A}_m \subset B_{L^\infty(\Omega)}(0,1-\delta).
$$
Next, by the second part of Theorem \ref{SPHR}, we infer from \eqref{PHI-hol} and \eqref{SP-2} that 
$$
\| S(t)\phi^\star\|_{C^\alpha(\overline{\Omega})} = 
\| S(t)\phi^\star \|_{C(\overline{\Omega})}+ \sup_{x,y\in \overline{\Omega}, x\neq y} \frac{|(S(t)\phi^\star)(x)-(S(t)\phi^\star)(y)|}{|x-y|^\alpha}
\leq 1-\delta + C_3=: N_4.
$$
Notice that $N_4$ is a universal constant which depends only on $N_2$, $\delta$, $m$ and the parameters of the system (namely, $F$, $J$, $\Omega$). Thus, since the constant are independent of $\phi^\star$, we conclude that $\mathcal{A}_m=S(1)\mathcal{A}_m\subset B_{C^\alpha(\overline{\Omega})}(0,N_4)$. The proof is complete.
\end{proof}

\textbf{Acknowledgment.} AG is a member of Gruppo Nazionale per
l'Analisi Ma\-te\-ma\-ti\-ca, la Probabilit\`{a} e le loro Applicazioni
(GNAMPA), Istituto Nazionale di Alta Matematica (INdAM). 




\begin{thebibliography}{99}

\bibitem{AW} \au{H. Abels, M. Wilke}, 
\ti{Convergence to equilibrium for the
Cahn-Hilliard equation with a logarithmic free energy}, \jou{Nonlinear Anal.}
\textbf{67} (2007), 3176--3193.

\bibitem{Ca} \au{J.W. Cahn}, \ti{On spinodal decomposition}, 
\jou{Acta
Metallurgica} \textbf{9} (1961), 795--801.

\bibitem{CH} \au{J.W. Cahn, J.E. Hilliard}, 
\ti{Free energy of a nonuniform
system. I. Interfacial free energy}, \jou{J. Chem. Phys.} \textbf{28}
(1958), 258--267.

\bibitem{CH2} \au{J.W. Cahn, J.E. Hilliard}, 
\ti{Spinodal decomposition: a
reprise}, \jou{Acta Metallurgica} \textbf{19} (1971), 151--161.



\bibitem{DST2020} {\au E. Davoli, L. Scarpa, L. Trussardi}, {\ti %
Nonlocal-to-local convergence of Cahn-Hilliard equations: Neumann boundary
conditions and viscosity terms}, {\jou Arch. Ration. Mech. Anal.} \textbf{239%
} (2021), 117--149.

\bibitem{DST2021} {\au E. Davoli, L. Scarpa, L. Trussardi}, {\ti Local
asymptotics for nonlocal convective Cahn-Hilliard equations with $W^{1,1}$
kernel and singular potential}, {\jou J.\ Differential Equations} \textbf{289%
} (2021), 35--58.

\bibitem{DD} {\au A. Debussche, L. Dettori}, {\ti On the Cahn-Hilliard
equation with a logarithmic free energy}, {\jou Nonlinear Anal.} \textbf{24}
(1995), 1491--1514.

\bibitem{Dung} {\au L. Dung}, 
\ti{Remarks on Hölder continuity for parabolic equations and convergence to global attractors}, \jou{Nonlinear Anal.} \textbf{41} (2000) 921--941.

\bibitem{E} {\au C.M. Elliott}, {\ti The Cahn-Hilliard model for the
kinetics of phase separation}, Mathematical models for phase change problems
(\'{O}bidos, 1988), 35--73, Internat. Ser. Numer. Math. \textbf{88}, Birkh%
\"{a}user, Basel, 1989.

\bibitem{FGG} {\au S. Frigeri, C.G. Gal, M. Grasselli}, {\ti Regularity
results for the nonlocal Cahn-Hilliard equation with singular potential and
degenerate mobility}, {\jou J.\ Differential Equations} \textbf{287} (2021),
295--328.

\bibitem{FG} {\au S. Frigeri, M. Grasselli}, \textit{Nonlocal
Cahn-Hilliard-Navier-Stokes systems with singular potentials}, {\jou Dyn.
Partial Differ. Equ.} \textbf{9} (2012), 273--304.

\bibitem{GZ} {\au H. Gajewski, K. Zacharias}, {\ti On a nonlocal phase
separation model}, {\jou J. Math. Anal. Appl.} \textbf{286} (2003), 11--31.

%


\bibitem{GGG} {\au C.G. Gal, A. Giorgini, M. Grasselli}, {\ti The nonlocal
Cahn-Hilliard equation with singular potential: well-posedness, regularity
and strict separation property}, {\jou J. Differential Equations} \textbf{263%
} (2017), 5253--5297.

\bibitem{GGG2} {\au C.G. Gal, A. Giorgini, M. Grasselli}, 
{\ti The separation property for 2D Cahn-Hilliard equations: Local, nonlocal and fractional energy cases}, {\jou  Discrete Contin. Dyn. Syst.}, online first (2023). Doi: \href{10.3934/dcds.2023010}{10.3934/dcds.2023010}.


\bibitem{GG2014} {\au C.G. Gal, M. Grasselli}, {\ti Longtime behavior of
nonlocal Cahn-Hilliard equations}, {\jou Discrete Contin. Dyn. Syst. Ser. A} 
\textbf{34} (2014), 145--179.

\bibitem{GS2022} {\au C.G. Gal, J. Shomberg}, {\ti Cahn-Hilliard equations
governed by weakly nonlocal conservation laws and weakly nonlocal particle
interactions}, {\jou Ann.\ Inst.\ H.\ Poincar\'{e} Anal. Non Lin\'{e}aire},
in press.

\bibitem{GL1997} {\au G. Giacomin, J.L. Lebowitz}, {\ti Phase segregation
dynamics in particle systems with long range interactions. I. Macroscopic
limits}, {\jou J.\ Statist.\ Phys.} \textbf{87} (1997), 37--61.

\bibitem{GL1998} {\au G. Giacomin, J.L. Lebowitz}, {\ti Phase segregation
dynamics in particle systems with long range interactions II: Interface
motion}, {\jou SIAM J. Appl. Math.} \textbf{58} (1998), 1707--1729.

\bibitem{GGM} \au{A. Giorgini, M. Grasselli, A. Miranville}, 
\ti{The
Cahn-Hilliard-Oono equation with singular potential}, 
\jou{Math. Models
Methods Appl. Sci.} \textbf{27} (2017), 2485--2510.

\bibitem{HW2021} \au{J.-N. He, H. Wu}, 
\ti{Global well-posedness of a
Navier-Stokes-Cahn-Hilliard system with chemotaxis and singular potential in
2D}, \jou{J. Differential Equations} \textbf{297} (2021), 47--80.

\bibitem{LSU} 
{\au O.A. Lady\v{z}enskaja, V.A. Solonnikov, N.N. Ural'ceva}, 
{\bk Linear
and quasilinear equations of parabolic type}, 
\eds{AMS}{Providence}{1968}.



\bibitem{LP2011} {\au S.-O. Londen, H. Petzeltov\'{a}}, {\ti Regularity and
separation from potential barriers for a non-local phase-field system}, {%
\jou J. Math. Anal. Appl.} \textbf{379} (2011), 724--735.

\bibitem{LP} {\au S.-O. Londen, H. Petzeltov\'{a}}, {\ti Regularity and
separation from potential barriers for the Cahn-Hilliard equation with
singular potential}, {\jou J. Evol. Equ.} \textbf{18} (2018), 1381--1393.

\bibitem{MM} {\au M. Marcus, V.J. Mizel}, {\ti Absolute continuity on tracks
and mappings of Sobolev spaces}, {\jou Arch. Rational Mech. Anal.} \textbf{45%
} (1972), 294--320.

\bibitem{MZ} \au{A. Miranville, S. Zelik}, 
\ti{Robust exponential attractors
for Cahn-Hilliard type equations with singular potentials}, 
\jou{Math.
Methods. Appl. Sci.} \textbf{27} (2004), 545--582.

\bibitem{Mbook} {\au A. Miranville}, {\bk The Cahn-Hilliard Equation: Recent
Advances and Applications}, CBMS-NSF Regional Conf. Ser. in Appl. Math. 
\textbf{95}, SIAM, Philadelphia, PA., 2019.

\bibitem{PZ}
\au{V. Pata, S. Zelik}, 
\ti{A result on the existence of global attractors for semigroups of closed operators}, 
\jou{Commun. Pure Appl. Anal.} 
\textbf{6} (2007) 481--486.

\bibitem{P}
\au{A. Poiatti}, \ti{The 3D strict separation property for the nonlocal Cahn-Hilliard equation with singular potential}, 
\jou{preprint 2022}.


\end{thebibliography}
\end{document}